\documentclass[10pt]{amsart}
\usepackage{amsmath,amssymb,amsthm}

\newcommand{\A}{\mathcal{A}}
\newcommand{\F}[2][\empty]{\mathcal{F}_{#1}(#2)}
\newcommand{\C}[1]{\mathcal{C}(#1)}
\newcommand{\W}{\mathcal{W}}
\newcommand{\pair}[1]{\langle#1\rangle}
\newcommand{\lex}{<_{lex}}
\newcommand{\I}{\mathcal{I}}
\newcommand{\cl}[2][X]{\mathrm{cl}_{#1}\!\left(#2\right)}

\newtheoremstyle{theorem}
     {11pt}
     {11pt}
     {}
     {}
     {\bfseries}
     {}
     {.5em}
     {\noindent\thmnumber{#2}. \thmname{#1}\thmnote{#3}}

\theoremstyle{theorem}

\newtheorem{thm}{Theorem}[section]
\newtheorem{lemma}[thm]{Lemma}

\newtheorem{remark}[thm]{Remark}

\newtheorem{ques}[thm]{Question}
\newtheorem{defi}[thm]{Definition}
\newtheorem{ex}[thm]{Example}
\newtheorem{claim}[thm]{Claim}
\newtheorem{fact}[thm]{Fact}

\title{A Whitney map onto the long arc}

\author{Rodrigo Hern\'andez-Guti\'errez}

\address{Department of Mathematics and Statistics, York University, Toronto, ON M3J 1P3, Canada}	
	\email{rodhdz@yorku.ca}

\date{\today}

\begin{document}

\begin{abstract}
In a recent paper, Garc\'ia-Velazquez has extended the notion of Whitney map to include maps with non-metrizable codomain and left open the question of whether there is a continuum that admits such a Whitney map. In this paper, we consider two examples of hereditarily indecomposable, chainable continua of weight $\omega_1$ constructed by Michel Smith; we show that one of them admits a Whitney function onto the long arc and the other admits no Whitney maps at all.
\end{abstract}

\maketitle

\section{Introduction}

All spaces in this note are assumed to be Tychonoff. For a continuum $X$, let $\C{X}$ be the hyperspace of subcontinua of $X$ with the Vietoris topology.

For metrizable continua, there exists the well-known notion of a Whitney map which, roughly speaking, gives the size of a closed subset relative to others. In this way, closed subsets of the same size are contained in a so-called Whitney level. Whitney maps and Whitney levels are widely used in the theory of metrizable continua, for details, see the book \cite{hyp-ill-nad}. 

The paper \cite{luismiguel-whitney} extends the definition of Whitney level (for $\C{X}$) so that it does not depend on Whitney maps and can be given on arbitrary continua (that is, continua that are not necessarily metrizable). The author of that paper studies whether some well-known non-metrizable continua have Whitney levels or not. He also defines Whitney maps for arbitrary continua and show that the examples he considers do not admit such generalized Whitney maps.

Let us give the definition of Whitney maps in its general form. An arc is a continuum $J$ with its topology given by a strict linear order relation $\lhd$ (equivalently, a continuum with only two non-cut points, but we will focus on the order relation). Contrasting with the case of the unit interval $[0,1]$, in general the reverse order of $\lhd$ may not be order-isomorphic to $\lhd$. So in this paper an arc will be a pair $\pair{J,\lhd}$ that consists on the topological space $J$ with the strict order relation $\lhd$ giving the topology.

\begin{defi}
If $X$ is a continuum, a Whitney map for $\C{X}$ is a continuous function $\mu:\C{X}\to \pair{J,\lhd}$, where $\pair{J,\lhd}$ is an arc and the following conditions hold: (a) $\mu(\{x\})=\min J$ for each $x\in X$; (b) $\mu(A)\lhd\mu(B)$ whenever $A,B\in\C{X}$ and $A\subsetneq B$; and (c) $\mu(X)=\max J$.
\end{defi}

Stone has constructed in \cite{stone-phd} a continuum of uncountable weight that admits a Whitney map to $[0,1]$. Thus it remains to show whether there exists a continuum that admits a Whitney map to an arc different from $[0,1]$; this was asked in \cite{luismiguel-whitney}. The purpose of this paper is to solve that question in the affirmative. 

We shall use an example of ``non-metric pseudoarc'' constructed by Michel Smith in \cite{smith2}. Recall that the \emph{long arc} is the space $L=(\omega_1\times[0,1))\cup\{\pair{\omega_1,0}\}$ with the topology given by the lexicographic ordering $\lex$. 

\begin{ex}\label{exyes}
There exists a chainable, hereditarily indecomposable continuum $X$  of weight $\omega_1$ that admits a Whitney map $\mu:\C{X}\to \pair{L,\lex\sp-}$, where $\lex\sp-$ is the reverse of $\lex$, but does not admit a Whitney map to a metric arc.
\end{ex}

Michel Smith also constructed another ``non-metric pseudoarc'' in \cite{smith1} but it turns out that this example does not work. We also give a proof for the sake of completeness.

\begin{ex}\label{exno}
There exists a chainable, hereditarily indecomposable continuum of weight $\omega_1$ that admits no Whitney map.
\end{ex}

We leave the following natural questions unsolved.

\begin{ques}
Does there exist a continuum $X$ that admits a Whitney map $\mu:\C{X}\to \pair{L,\lex}$?
\end{ques}

\begin{ques}
Given an arc $\pair{J,\lhd}$, is it possible to find a continuum $X$ that admits a Whitney map $\mu:\C{X}\to\pair{J,\lhd}$?
\end{ques}

\section{Definitions and conventions}

A continuum is non-degenerate if it contains more than one point. Let $X$ be a continuum. An \emph{order arc} in $\C{X}$ is a continuum $\A\subset\C{X}$ such that if $A,B\in\A$ then either $A\subset B$ or $B\subset A$. Clearly, with the subspace topology, an order arc is a topological arc. A \emph{long order arc} in $\C{X}$ is an order arc $\A$ such that $\A\cap\F[1]{X}\neq\emptyset$ and $X\in\A$. The following observation is obvious but essential.

\begin{remark}\label{obvious}
If a continuum $X$ admits a Whitney function onto an arc $\pair{J,\lhd}$, then any long order arc is order-isomorphic (thus, homeomorphic) to $\pair{J,\lhd}$.
\end{remark}

\begin{defi}(\cite{luismiguel-whitney})\label{definitionlevel}
Let $X$ be a continuum. A compact subset $\W\subset\C{X}$ is a Whitney level if the three following conditions hold:
\begin{itemize}
\item[(a)] either $\W=\F[1]{X}$ or $\W\cap\F[1]{X}=\emptyset$,
\item[(b)] if $A,B\in\W$ and $A\neq B$ then $A\not\subset B$ and $B\not\subset A$, and
\item[(c)] $\W$ intersects each long order arc.
\end{itemize}
\end{defi}

We will assume that the reader is familiar with the concepts of hereditarily indecomposable continuum, composant and chainable continuum (in the general, not necessarily metric setting). We will use the following results that are known for metrizable continua and can be proved in the general setting: there are order arcs between any two subcontinua that are comparable (\cite[Section 14]{hyp-ill-nad}) and the hyperspace of subcontinua of a hereditarily indecomposable continuum is uniquely arcwise connected (\cite[Section 18]{hyp-ill-nad}).

\begin{lemma}\label{confluenttohi}\cite[Theorem 4]{cook}
If $f:X\to Y$ is a surjective continuous function between continua, $X$ is metrizable and $Y$ is hereditarily indecomposable, then $f$ is confluent. This means that if $A\in\C{Y}$ and $B$ is a component of $f\sp\leftarrow[A]$, then $f[B]=A$.
\end{lemma}

The pseudoarc is the continuum characterized by being the only hereditarily indecomposable and chainable metrizable continuum. See \cite{pseudoarc} for an overview on the pseudoarc.

We will also assume the reader's familiarity with inverse sequences and inverse limits (of length an arbitrary ordinal). See \cite{chigogidze} or \cite[2.5]{eng} for introductions in the general setting.

We will write $\pair{X_\alpha,f_\alpha\sp\beta,\lambda}$ for an inverse sequence of length the limit ordinal $\lambda$, where $X_\alpha$ are the base spaces and $f_\alpha\sp\beta:X_\beta\to X_\alpha$ are the bonding functions. The inverse limit will be written as $\lim_{\leftarrow}{\pair{X_\alpha,f_\alpha\sp\beta,\lambda}}=\pair{X,\pi_\alpha}_{\lambda}$ and consists on the limit space $X$ and a projection $\pi_\alpha:X\to X_\alpha$ for each $\alpha<\lambda$. Concretely, in this situation the limit space may be constructed as
$$
X=\big\{x\in\prod\{X_\alpha:\alpha<\lambda\}:\forall\alpha<\beta<\lambda\ [x(\alpha)=f_\alpha\sp\beta(x(\beta))]\big\},
$$
and the projections are the corresponding restrictions of projections to the factor spaces of the product. An inverse sequence $\pair{X_\alpha,f_\alpha\sp\beta,\lambda}$ is continuous if every time $\gamma<\lambda$ is a limit ordinal, then $\pair{X_\gamma,f_\alpha\sp\gamma}_\gamma=\lim_{\leftarrow}{\pair{X_\alpha,f_\alpha\sp\beta,\gamma}}$. 

The following is a result by Gentry and Pol according to \cite[6.3.16]{eng}.

\begin{lemma}\label{lemmamonotone}
If an inverse sequence of continua has all its bonding maps monotone, then the projections from the inverse limit are also monotone.
\end{lemma}

The following result is easy to prove and well-known in the metric setting.

\begin{lemma}\label{limitpseudoarcs}
The inverse limit of a sequence of non-degenerate hereditarily indecomposable (chainable) continua with surjective bonding maps is non-degenerate and hereditarily indecomposable (respectively, chainable). In particular, every inverse limit of a countable sequence of pseudoarcs with surjective bonding maps is a pseudoarc.
\end{lemma}

\section{Example \ref{exyes}}

As mentioned before, we will use an example constructed by Smith in \cite{smith2}. Smith's example is homogeneous and hereditarily equivalent. However, we do not need these properties so for the sake of simplicity, we give a simplified version of Smith's example.

\begin{lemma}\label{decomposition}
Let $P$ be the pseudoarc. Then there exists a continuous function $\pi:P\to P$ such that $\{\pi\sp\leftarrow(x):x\in P\}$ is a continuous decomposition of $P$ into pseudoarcs.
\end{lemma}
\begin{proof}
Let $\W$ be a non-trivial Whitney level of $P$. It is known that $\W$ is hereditarily indecomposable and chainable (see Theorems 37.4 and 44.1 in \cite{hyp-ill-nad}). Thus, there is a homeomorphism $h:\W\to P$. Notice that each element of $\W$ is a pseudoarc and any two elements of $\W$ are disjoint. Let $\pi:P\to P$ be defined so that $\pi(x)=h(K)$ where $x\in K\in\W$.
\end{proof}

The construction of the example is in $\omega_1$ steps. Construct a continuous inverse sequence $\pair{X_\alpha,\pi_\alpha\sp\beta,\omega_1}$ such that:
\begin{itemize}
\item[(i)] $X_0$ is a singleton and if $0<\alpha<\omega_1$, $X_\alpha$ is homeomorphic to the pseudoarc; and
\item[(ii)] for each $\alpha<\omega_1$, the function $\pi_\alpha\sp{\alpha+1}:X_{\alpha+1}\to X_\alpha$ is continuous and $\{(\pi_\alpha\sp{\alpha+1})\sp\leftarrow(x):x\in X_\alpha\}$ is a continuous decomposition of $X_{\alpha+1}$ into pseudoarcs.
\end{itemize}

Notice that the condition (ii) is possible by Lemma \ref{decomposition} and the sequence can be made continuous thanks to Lemma \ref{limitpseudoarcs}. Let $\pair{X,\pi_\alpha}_{\omega_1}=\lim_{\leftarrow}{\pair{X_\alpha,\pi_\alpha\sp\beta,\omega_1}}$. Then $X$ is hereditarily indecomposable and chainable by Lemma \ref{limitpseudoarcs}. 

To simplify what follows, we adopt the following notation: let $X_{\omega_1}=X$; $\pi_{\omega_1}:X\to X$ is the identity; if $\alpha\leq\omega_1$ and $x\in X$ let $x_\alpha=\pi_\alpha(x)$. The following follows because the sequence is continuous.

\vskip12pt
\noindent($\ast$)
If $x\in X$ and $\gamma\leq\omega_1$ is a limit ordinal, then $\pi_\gamma\sp\leftarrow(x_\gamma)=\bigcap\{\pi_\alpha\sp\leftarrow(x_\alpha):\alpha<\gamma\}$.
\vskip12pt

For each $\alpha<\omega_1$, let $\W_\alpha=\{\pi_\alpha\sp\leftarrow(x):x\in X_\alpha\}$, by Lemma \ref{lemmamonotone}, $\W_\alpha\subset\C{X}$ and is pairwise disjoint. Notice that $\W_\beta$ refines $\W_\alpha$ whenever $\alpha<\beta\leq\omega_1$.

\begin{claim}
If $A\in\C{X}\setminus\F[1]{X}$ then there exists $\gamma<\omega_1$, $B\in\W_\gamma$ and $C\in\W_{\gamma+1}$ such that $C\subset A\subset B$.
\end{claim}

Let $x\in A$. By $(\ast)$, $A$ intersects each continuum in $\{\pi_\alpha\sp\leftarrow(x_\alpha):\alpha\leq\omega_1\}$. Notice that $\pi_0\sp\leftarrow(x_0)=X\supset A$ and $\pi_{\omega_1}\sp\leftarrow(x_{\omega_1})=\{x\}\not\supset A$. Let $\beta$ be the minimal ordinal $\alpha\leq\omega_1$ such that $A\not\subset\pi_\alpha\sp\leftarrow(x_\alpha)$. 

First, notice that $\beta$ cannot be a limit ordinal. Otherwise, since $A\subset\pi_\alpha\sp\leftarrow(x_\alpha)$ for each $\alpha<\beta$, by $(\ast)$ we obtain that $A\subset\pi_\beta\sp\leftarrow(x_\beta)$. Also, $\beta\neq0$ so $\beta=\gamma+1$ for some $\gamma<\omega_1$. Since $X$ is hereditarily indecomposable, $A\supset\pi_\beta\sp\leftarrow(x_\beta)$. Then $\gamma$, $B=\pi_\gamma\sp\leftarrow(x_\gamma)$ and $C=\pi_\beta\sp\leftarrow(x_\beta)$ witness the Claim.
\vskip15pt

Consider the function $\Pi_\alpha:\C{X}\to\C{X_{\alpha+1}}$ defined by $\Pi_\alpha(A)=\pi_{\alpha+1}[A]$ for all $A\in\C{X}$, it is well-known that this construction produces a continuous function. The following properties of this function can be easily checked.

\begin{claim}\label{propertiesPi}
Let $\alpha<\omega_1$. Then $\Pi_\alpha$ has the following properties:
\begin{itemize}
\item[(1)] if $A,B\in\C{X}$ and $A\subset B$ then $\Pi_\alpha(A)\subset\Pi_\alpha(B)$,
\item[(2)] $(\Pi_\alpha)\sp\leftarrow(A)$ is a singleton exactly when $A\notin\F[1]{X_{\alpha+1}}$,
\item[(3)] $\W_{\alpha}=(\Pi_\alpha)\sp\leftarrow[\{(\pi_\alpha\sp{\alpha+1})\sp\leftarrow(x):x\in X_{\alpha}\}]$,
\item[(4)] $(\Pi_\alpha)\sp\leftarrow[\F[1]{X_{\alpha+1}}]=\{A\in\C{X}:\exists B\in\W_{\alpha+1}\ (A\subset B)\}$,
\item[(5)] $\Pi_\alpha[\W_{\alpha+1}]=\F[1]{X_{\alpha+1}}$ and
\item[(6)] $\Pi_\alpha\!\!\restriction_{\W_{\alpha+1}}$ is one-to-one.
\end{itemize}
\end{claim}

Using this, let us prove the following.

\begin{claim}\label{claimfirstlevels}
For each $\alpha\leq\omega_1$, $\W_\alpha$ is a Whitney level of $\C{X}$.
\end{claim}

The Claim is obvious if $\alpha\in\{0,\omega_1\}$ so let us assume this is not the case. We have already mentioned that $\W_\alpha\subset\C{X}$. By (3) in Claim \ref{propertiesPi}, we obtain that $\W_\alpha$ is closed. Properties (a) and (b) from Definition \ref{definitionlevel} follow easily. Now let $\A$ be a long order arc. There is $x\in X$ such that $\{x\}\in\A$ so $B=\pi_\alpha\sp\leftarrow(x_\alpha)\in\W_\alpha$. So it is possible to construct a long order arc $\A\sp\prime$ from $\{x\}$ to $X$ that passes through $B$. But the hyperspace of a hereditarily indecomposable continuum is uniquely arcwise connected so $\A\sp\prime=\A$. This proves (c) in Definition \ref{definitionlevel} and the Claim.
\vskip15pt

Now, for each $\alpha<\omega_1$, let 
$$
\I(\alpha)=\{A\in\C{X}:\exists B\in\W_\alpha, \exists C\in\W_{\alpha+1}\ (C\subset A\subset B)\}
$$
which is the interval of all subcontinua between $\W_\alpha$ and $\W_{\alpha+1}$. By Claim \ref{claimfirstlevels}, it is not hard to prove that $\I(\alpha)$ is closed, thus compact. Define $h_\alpha:\I(\alpha)\to\C{X_{\alpha+1}}$ to be $h_\alpha=\Pi_\alpha\!\!\restriction_{\I(\alpha)}$. By Claim \ref{propertiesPi} we obtain that $h_\alpha$ is an embedding that preserves the strict order relation $\subset$. 

Let $J=[0,2]$ if $\alpha\neq 0$ and $J=[0,1]$ if $\alpha=0$. By the famous result on extension of Whitney maps (\cite[Theorem 16.10]{hyp-ill-nad}), there exists a Whitney map $\mu_\alpha:\C{X_{\alpha+1}}\to\pair{J,<}$ such that $\mu\sp\leftarrow(1)=\{(\pi_\alpha\sp{\alpha+1})\sp\leftarrow(x):x\in X_\alpha\}$. Then it follows that $h_\alpha\circ\mu_\alpha:\I(\alpha)\to[0,1]$ is a continuous function that preserves the strict order relation, so it is a partial Whitney map for the set $\I(\alpha)$.

Finally, let us define the desired Whitney map $\mu:\C{X}\to\pair{L,\lex\sp-}$ by pasting the partial Whitney maps we have obtained in the following way:
$$
\mu(A)=\left\{
\begin{array}{ll}
\pair{\omega_1,0}, & \textrm{ if } A\in\F[1]{X},\\
\pair{\alpha,1-(h_\alpha\circ\mu_\alpha)(A)}, & \textrm{ if } A\in\I(\alpha)\setminus\W_{\alpha+1}\textrm{ for some }\alpha<\omega_1.
\end{array}
\right.
$$

The proof that this is a Whitney map follows easily from all the work done before. It remains, of course, to show that $X$ admits no Whitney maps to $\pair{[0,1],<}$ and this follows from Remark \ref{obvious}. Also, as a last remark, the weight of $X$ is at most $\omega_1$ by construction and must be uncountable because it admits no Whitney maps to a metric arc; as a consequence, its weight is $\omega_1$.

\section{Example \ref{exno}}

In this section we will explore Smith's example from \cite{smith1}. We will just give the main ideas of the construction, the proofs can be read in the original paper. This gives us an example of a ``non-metric pseudoarc'' that admits no Whitney maps. 

\begin{lemma}\cite[Theorem 1]{smith1}\label{lemmaretraction}
Let $P\subsetneq Q$ pseudoarcs and $p\in P$. Then there is a continuous function $r:Q\to P$ such that $r(x)=x$ for each $x\in P$ (that is, it is a retraction) and $r(y)$ is in the same composant as $p$, whenever $y\in P\setminus Q$.
\end{lemma}

Recursively, it is possible to construct a continuous inverse sequence $\pair{X_\alpha,f_\alpha\sp\beta,\omega_1}$, and sequences $\{i_\alpha\sp\beta:\alpha<\beta<\omega_1\}$, $\{K_\alpha:\alpha<\omega_1\}$, where:
\begin{itemize}
\item[(i)] for every $\alpha<\omega_1$, $X_\alpha$ is a pseudoarc and $K_\alpha$ is a composant of $X_\alpha$,
\item[(ii)] for every $\alpha<\beta<\omega_1$, $i_\alpha\sp\beta:X_\alpha\to X_\beta$ is an embedding and $i_\alpha\sp\beta[X_\alpha]\cap K_\beta=\emptyset$,
\item[(iii)] for every $\alpha<\beta<\omega_1$, $f_\alpha\sp\beta\circ i_\alpha\sp\beta$ is the identity function on $X_\alpha$, and
\item[(iv)] for every $\alpha<\omega_1$, $f_\alpha\sp{\alpha+1}[X_{\alpha+1}\setminus i_\alpha\sp{\alpha+1}[X_\alpha]]\subset K_\alpha$,
\end{itemize}

Notice that conditions (ii) and (iii), for $\beta=\alpha+1$, tell us that each function $f_\alpha\sp{\alpha+1}$ works as a retraction and by condition (iv), one part of the domain is mapped to a composant of our choice; this is of course an application of Lemma \ref{lemmaretraction}. The limit steps are possible by Lemma \ref{limitpseudoarcs}. Let $\pair{X,\pi_\alpha}_{\omega_1}=\lim_{\leftarrow}{\pair{X_\alpha,f_\alpha\sp\beta,\omega_1}}$. The fact that $X$ is hereditarily indecomposable, chainable and of weight $\omega_1$ can be checked in a similar way as for Example \ref{exyes}.

Again, take the following conventions: $X_{\omega_1}=X$; $\pi_{\omega_1}:X\to X$ is the identity; if $\alpha\leq\omega_1$ and $x\in X$ let $x_\alpha=\pi_\alpha(x)$.

Now that we have completed the construction, we may assume that the embeddings $i_\alpha\sp\beta:X_\alpha\to X_\beta$ are in fact inclusions. More generally, let us assume that $X_\alpha\subset X_\beta$ whenever $\alpha\leq\beta\leq\omega_1$. In this way, both the functions $f_\alpha\sp\beta$ and the functions $\pi_\alpha$ are retractions.

To prove that there are no Whitney maps for $\C{X}$, by Remark \ref{obvious}, it is enough to find two long order arcs that are not homeomorphic. In order to find these two long order arcs, we will use two different composants of $X$: $C_0=\bigcup\{X_\alpha:\alpha<\omega_1\}$ and $C_1=X\setminus C_0$.

\begin{fact}\cite{smith1}\label{easycomposant}
If $\gamma\leq\omega_1$ is a limit ordinal, then $\bigcup\{X_\alpha:\alpha<\gamma\}$ is a composant of $X_\gamma$.
\end{fact}

Then $C_0$ is a composant of $X$, and we can construct an order arc in fact. Choose $p\in X_0$ arbitrarly and let $\A_0$ be the unique long order arc that goes from $\{p\}$ to $X$. Since $X$ is hereditarily indecomposable, $\C{X}$ is uniquely arcwise connected  and $X_\alpha\in\A_0$ for each $\alpha<\omega_1$. Notice that by Fact \ref{easycomposant}, whenever $\gamma\leq\omega_1$ is a limit ordinal, $X_\gamma=\cl[X_\gamma]{\bigcup\{X_\alpha:\alpha<\gamma\}}$. Then it is not hard to see the following.

\begin{claim}\label{easyorderarc}
The set $\{X_\alpha:\alpha<\omega_1\}$ is a cofinal set of order type $\pair{\omega_1,<}$ in the long order arc $\pair{\A_0,\subset}$.
\end{claim}

We have to use the second composant to constuct a different order arc. 

\begin{fact}\cite{smith1}
$C_1$ is the union of an increasing collection of continua of order $\pair{\omega,<}$.
\end{fact}

The result above was quite unexpected for the author of this note and shows how this metric ``anomally'' is transfered to the composant that we cannot control due to the limiting process. In what follows, we shall sketch Smith's proof of this beatiful fact; the details can be checked in \cite{smith1}.

Notice that by the choice of the composants $K_\alpha$, the following equality holds.
$$
C_1=\{x\in X:\forall \alpha<\omega_1\ (x_\alpha\in K_\alpha)\}.
$$

Fix some $q\in C_1$. In a metrizable continuum, every composant is the union of a countable collection of continua (\cite[11.14]{nadler}). Thus, there exists $\{M(0,n):n<\omega\}\subset\C{X_0}$ such that $q_0\in M(0,0)$, $M(0,n)\subset M(0,n+1)$ for all $n<\omega$ and $K_0=\bigcup\{M(0,n):n<\omega\}$.

Now fix $n<\omega$, recursively construct a sequence of continua $\{M(n,\alpha):\alpha\leq\omega_1\}$ in the following way.
\begin{itemize}
\item If $\alpha<\omega_1$, let $M(n,\alpha+1)$ be the component of $(f_\alpha\sp{\alpha+1})\sp\leftarrow[M(n,\alpha)]$ that contains $q_{\alpha+1}$.
\item If $\gamma\leq\omega_1$ is a limit ordinal, then $M(n,\gamma)=\bigcap\{(f_\alpha\sp\gamma)\sp\leftarrow[M(n,\alpha)]:\alpha<\gamma\}$.
\end{itemize}

By Lemma \ref{confluenttohi}, $f_\alpha\sp{\alpha+1}$ is confluent so $f_\alpha\sp{\alpha+1}[M(n,\alpha+1)]=M(n,\alpha)$ for each $\alpha<\omega_1$. Moreover, when $\gamma\leq\omega_1$ is a limit ordinal, by the properties of inverse limits, $p_\gamma\in M(n,\gamma)$ and $M(n,\gamma)$ is a continuum (see \cite[1.2.4]{chigogidze} and \cite[6.1.20]{eng}).

By induction, using the properties of inverse limits, it is possible to prove that $K_\alpha=\bigcup\{M(n,\alpha):n<\omega\}$ for each $\alpha<\omega_1$. The last step is proving that $C_1=\bigcup\{M(n,\omega_1):n<\omega\}$. Notice that in the countable steps, something like this is not possible because pseudoarcs have $\mathfrak{c}$ composants.

So let $x\in C_1$. For each $\alpha<\omega_1$, $x(\alpha)\in K_\alpha$ so there is $k(\alpha)<\omega$ such that $x(\alpha)\in M(k(\alpha),\alpha)$. So there exists some uncountable $S\subset\omega_1$ and $k<\omega$ such that $k(\alpha)=k$ for each $\alpha\in S$. But every uncountable subset of $\omega_1$ is cofinal in $\omega_1$ so by the recursive definition of the sets $M(k,\alpha)$ it is not hard to see that $S=\omega_1$. In this same way, it is possible to show that $x\in M(k,\omega_1)$.

This completes the fact that $C_1$ is another composant which is the increasing union of a countable collection of continua. Then there exists a long order arc $\A_1$ that passes through all these continua.

\begin{claim}\label{hardorderarc}
The collection $\{M(n,\omega_1):n<\omega\}$ is a cofinal sequence of order type $\pair{\omega,<}$ in the long order arc $\pair{\A_1,\subset}$.
\end{claim}

A linearly ordered set cannot have a cofinal sequence of order type $\omega_1$ and another of order type $\omega$. So by Claims \ref{easyorderarc} and \ref{hardorderarc}, the long order arcs $\A_0$ and $\A_1$ are not order isomorphic. Thus, there is no Whitney map on $\C{X}$.

\end{document}